\documentclass[reqno, 12pt]{amsart}
\pdfoutput=1
\makeatletter
\let\origsection=\section \def\section{\@ifstar{\origsection*}{\mysection}} 
\def\mysection{\@startsection{section}{1}\z@{.7\linespacing\@plus\linespacing}{.5\linespacing}{\normalfont\scshape\centering\S}}
\makeatother        
\usepackage{amsmath,amssymb,amsthm}    
\usepackage{mathabx}\changenotsign    
\usepackage{mathrsfs}
\usepackage{bbm, accents} 

\usepackage{xcolor}  	
\usepackage[backref]{hyperref}
\hypersetup{
	colorlinks,
    linkcolor={red!60!black},
    citecolor={green!60!black},
    urlcolor={blue!60!black},
}

\usepackage{bookmark}

\usepackage[abbrev,msc-links,backrefs]{amsrefs} 
\usepackage{doi}

\renewcommand{\PrintDOI}[1]{\doi{#1}}

\usepackage[T1]{fontenc}
\usepackage{lmodern}

\usepackage[english]{babel}

\theoremstyle{plain}
\newtheorem{fact}{Fact}[section]
\newtheorem{thm}[fact]{Theorem}
\newtheorem{prop}[fact]{Proposition}
\newtheorem{cor}[fact]{Corollary}
\newtheorem{claim}[fact]{Claim}
\newtheorem{conj}[fact]{Conjecture}

\theoremstyle{definition}
\newtheorem{rem}[fact]{Remark}

\usepackage{geometry}
\numberwithin{equation}{section}

\let\theta=\vartheta
\let\rho=\varrho
\let\phi=\varphi

\usepackage{enumitem}

\def\tn{\textnormal}                           
\def\lra{\longrightarrow}                         
\def\gs{\geqslant}                             
\def\ls{\leqslant}

\def\vt{\vartheta}
\def\GG{{\mathcal G}}

\theoremstyle{plain}

\begin{document}

\title[The Clique Density Theorem]{The Clique Density Theorem}

\author[Christian Reiher]{Christian Reiher}
\address{Fachbereich Mathematik, Universit\"at Hamburg, Hamburg, Germany}
\email{Christian.Reiher@uni-hamburg.de}

\keywords{structural Ramsey theory, Ramsey classes, partite construction}
\subjclass[2010]{Primary 05D10, Secondary 05C55}

\begin{abstract}
Tur\'{a}n's theorem is a cornerstone of extremal graph theory. It asserts that for any integer 
$r \gs 2$ every graph on $n$ vertices with more than ${\tfrac{r-2}{2(r-1)}\cdot n^2}$ edges 
contains a clique of size $r$, i.e., $r$ mutually adjacent vertices. The corresponding
extremal graphs are balanced $(r-1)$-partite graphs.

The question as to how many such $r$-cliques appear at least in any~$n$-ver\-tex graph 
with $\gamma n^2$ edges has been intensively studied in the literature. In particular,
Lov\'{a}sz and Simonovits conjectured in the 1970s that asymptotically the best possible 
lower bound is given by the complete multipartite graph with $\gamma n^2$ edges in which
all but one vertex class is of the same size while the remaining one may be smaller.

Their conjecture was recently resolved for $r=3$ by Razborov and for $r=4$ by Nikiforov.
In this article, we prove the conjecture for all values of~$r$. 
\end{abstract}

\keywords{extremal graph theory, clique density problem}
\subjclass[2000]{Primary: 05C35, Secondary: 05C22.}

\maketitle

\section{Introduction}

Extremal graph theory was initiated as a separate subarea of combinatorics 
by P.~Tur\'{a}n in~1941. 
In his famous article~\cite{Turan} the following problem is solved: Given integers
${n\gs r\gs 3}$, what is the maximum number of edges that an $n$-vertex graph may have without
containing a clique of size $r$, i.e., $r$ vertices any two of which are connected by an edge.
It turns out that there is a unique extremal graph for this problem, which is the 
complete $(r-1)$-partite graph with the property that the sizes of any two of its vertex 
classes differ by at most one. In particular if a graph with $n$ vertices has more than 
$\tfrac{r-2}{2(r-1)}\cdot n^2$ edges, then it needs to contain an $r$-clique, and for fixed 
$r$ the constant $\tfrac{r-2}{2(r-1)}$ appearing in this statement is sharp.

Given this result, one may ask how many $r$-cliques are guaranteed to exist in
graphs having more edges. That is, given an integer $r\gs 3$ and a real number number 
$\gamma>\tfrac{r-2}{2(r-1)}$ we want to know: what the minimum number of $r$-cliques 
appearing in an $n$-vertex graph with at least $\gamma n^2$ edges? Following the 
work~\cite{Goodman},~\cite{Momo},~\cite{Nost}, and~\cite{Had-Nik}, 
a general conjecture was formulated by Lov\'{a}sz and Simonovits in~\cite{LovSim}.  
The guiding idea behind their conjecture is that, up to some ``rounding errors'', there
should be for each such case an extremal graph that is again complete and multipartite.
Moreover, all its vertex classes should be of the same size, except for one that may be 
smaller. 

Now consider such an $(s+1)$-partite graph, $s$ of whose vertex classes have size   
$\tfrac{n(1+\alpha)}{s+1}$, so that the remaining class contains just 
$\tfrac{n(1-s\alpha)}{s+1}$ vertices, where $\alpha\in\bigl[0, \tfrac{1}{s}\bigr]$.
A short calculation discloses that for $\gamma=\tfrac{s}{2(s+1)}(1-\alpha^2)$ this graph has 
$\gamma n^2$ edges and that the number of its $r$-cliques is 
\[
\frac 1{(s+1)^r}\binom{s+1}{r}(1+\alpha)^{r-1}\bigl(1-(r-1)\alpha\bigr)\cdot n^r\,.
\]

We thus arrive at the following clique density conjecture due to Lov\'{a}sz and Simonovits.

\begin{conj} \label{CDC}
If $r\gs 3$ and $\gamma\in\left[0, \tfrac12\right)$, then every graph on $n$ vertices with at least $\gamma n^2$ edges contains at least
\[
\frac 1{(s+1)^r}\binom{s+1}{r}(1+\alpha)^{r-1}\bigl(1-(r-1)\alpha\bigr)\cdot n^r
\]
cliques of size $r$, where $s\gs 1$ is an integer with 
$\gamma\in\bigl[\tfrac{s-1}{2s}, \tfrac{s}{2(s+1)}\bigr]$ and 
$\alpha\in\bigl[0, \tfrac 1s\bigr]$ is implicitly defined by $\gamma=\tfrac{s}{2(s+1)}(1-\alpha^2)$.
\end{conj}

The aim of this article is to prove this conjecture.

\begin{rem}\label{rem:cd} (1) If $s\ls r-2$, then the binomial coefficient $\binom{s+1}{r}$ vanishes, 
which means that, in accordance with Tur\'{a}n's construction, the clique density conjecture 
does not predict the existence of any $r$-cliques for $\gamma\ls \tfrac{r-2}{2(r-1)}$.  

(2) If $r\gs 3$ and $\gamma>\tfrac{r-2}{2(r-1)}$ are fixed while $n$ tends to infinity,
the clique density conjecture guarantees in particular $\Omega_\gamma(n^r)$ many 
$r$-cliques. This phenomenon is known as ``supersaturation'' in the literature. 
Due to the above discussion it should be clear that the precise factor occurring in the
clique density conjecture is optimal for fixed $r$ and $\gamma$.

(3) If $\gamma\in\left[0, \tfrac12\right)$ is not of the form $\gamma=\tfrac{t}{2(t+1)}$ 
for some positive integer $t$, then there is a unique way of choosing the pair $(s, \alpha)$
as above. On the other hand, if $\gamma=\tfrac{t}{2(t+1)}$ has this form, there are two
legitimate choices for this pair, namely $(t, 0)$ and $(t+1, \tfrac 1{t+1})$. 
Yet it is not hard to verify that both of them lead to the same lower bound of
\[
\frac 1{(t+1)^r}\binom{t+1}{r}\cdot n^r
\]
on the number of $r$-cliques. 
\end{rem}

We would like to conclude this introduction with some historical comments: 
Let $G$ be any $n$-vertex graph with at least $\gamma n^2$ edges.
Goodman~\cite{Goodman} proved that $G$ contains at least $\tfrac 13\gamma(4\gamma-1)n^3$ 
triangles. 
This fact was also obtained by Nordhaus and Stewart~\cite{Nost}. 
Moon and Moser~\cite{Momo} proved that if $\gamma\gs\tfrac 13$, then $G$ also contains
at least 
\[
\tfrac{1}{12}\gamma (4\gamma-1)(6\gamma-2)n^4
\]
cliques of size four and stated without proof that similarly, 
if $\gamma\gs\tfrac{r-2}{2(r-1)}$, then~$G$ contains at least
\begin{equation}\label{eq:conv}
\frac 1{r!}\cdot 2\gamma(4\gamma-1)(6\gamma-2)\cdot\ldots\cdot
\bigl(2(r-1)\gamma-(r-2)\bigr)\cdot n^r
\end{equation}
cliques of size $r$. This was subsequently shown by Kad{\v{z}}iivanov and
Nikiforov~\cite{Had-Nik}. It may be observed that the factor appearing in front of $n^r$
in~\eqref{eq:conv} is a convex function of $\gamma$ for $\gamma\gs\tfrac{r-2}{2(r-1)}$.
We will return to this ``convex  bound'' in Section~\ref{sec:convex}. It is not hard to check that
if $\gamma=\frac{t}{2(t+1)}$ with $t\gs r-2$ is one of the ``critical'' values from 
Remark~\ref{rem:cd}(3) then the convex bound is optimal and yields the same prediction 
on the number of $r$-cliques in $G$ as the clique density conjecture does. Between these
critical values, however, the optimal bound is piecewise concave (we will check this in 
Section~\ref{sec:anna}). Thus the piecewise linear function interpolating between these
critical values should also be a lower bound on the number of $r$-cliques in $G$
and this has in fact been shown by Bollob\'{a}s~\cite{BolRaz}.
For an alternative proof of a more general result we refer to~\cite{SchTho}.
  
The case $r=3$ and $\gamma\in\bigl[\tfrac14, \tfrac13\bigr]$ of the clique density conjecture 
was studied by Fisher~\cite{Fish} (see also~\cite{GolSan} and~\cite{Csik}*{Remark 3.3}). 
An altogether different approach to this 
case has later been given by Razborov in the fifth section of \cite{RazF}. 
The proof described there  is based on what one might 
call the ``differential calculus of flag algebra homomorphisms'', which in turn constitutes 
an important part of Razborov's flag algebraic investigations. 
Shortly afterwards, Razborov~\cite{RazT} used this calculus for resolving the case $r=3$ of 
Conjecture~\ref{CDC} for all $\gamma$.

The next important step is due to Nikiforov~\cite{Nik} who found an independent proof for
the case $r=3$ and settled the case $r=4$ as well. In~\cite{Nik} Nikiforov suggests 
to study the clique density problem in the setting of ``weighted graphs'' and we will 
follow this idea in the sequel. The problem thus translates into a question about
polynomial forms and we follow Nikiforov in applying differential techniques to these
forms. Moreover, we use some of Razborov's ideas in this framework. 

\section{Weighted Graphs}

Given a set $X$ and a positive integer $r$, we use $X^{(r)}$ to denote the collection of 
all \hbox{$r$-element} subsets of $X$. Also, if $n$ refers to a positive integer, then $[n]$ is, 
by definition, shorthand for $\{1, 2, \ldots, n\}$. By a {\it weighted graph of order $n$}, 
we mean a pair consisting of a sequence $(x_1, x_2, \ldots, x_n)$ of $n$ nonnegative real 
numbers the sum of which is equal to~$1$ and a function $a\colon [n]^{(2)}\lra [0, 1]$. 
In such situations, if $e=\{i, j\}\in [n]^{(2)}$, we will often write $a_e$ or $a_{ij}$ in
place of~$a(e)$.

Whenever $\GG$ is such a weighted graph of order $n$ and $r$ is a positive integer, 
we define the {\it $r$-clique density} of $\GG$ to be
\[
\GG(K_r)=\sum_{M\in[n]^{(r)}}\prod_{e\in M^{(2)}}a_e\prod_{i\in M}x_i\,.
\]
Notice for instance that
\[
\GG(K_1)=x_1+x_2+\ldots+x_n=1\,.
\]

The following weighted variant of Conjecture~\ref{CDC} is, as we are soon going to see,
equivalent to it. To the best of our knowledge, it has for the first time been formulated 
explicitly by Nikiforov~\cite{Nik}. 

\begin{claim} \label{Cl21}
Let $r\gs 3$ denote an integer and let $\GG$ be a weighted graph. 
Suppose that a positive integer $s$ and a real number $\alpha\in\left[0, \tfrac 1s\right]$ 
are chosen in such a way that 
\[
	\GG(K_2)=\tfrac{s}{2(s+1)}(1-\alpha^2)\,.
\]
Then
\[
\GG(K_r)\gs \frac 1{(s+1)^r}\binom{s+1}{r}(1+\alpha)^{r-1}\bigl(1-(r-1)\alpha\bigr)\,.
\]
\end{claim}

To see that this indeed entails Conjecture~\ref{CDC}, take a graph $G$ and an integer 
$r\gs 3$, label the vertices of $G$ arbitrarily 
as $\{v_1, v_2, \ldots, v_n\}$, and construct a weighted graph $\GG$ of order~$n$ by 
the stipulations $x_1=x_2=\ldots=x_n=\tfrac 1n$ and
\[
a_{ij}=
\begin{cases}
1 & \textnormal{ if $v_i$ and $v_j$ are joined by an edge of $G$,}  \cr
0 & \textnormal{ otherwise }  \cr
\end{cases}
\]
for all $\{i, j\}\in[n]^{(2)}$. Plainly $G$ has exactly $\GG(K_2)\cdot n^2$ edges and 
$\GG(K_r)\cdot n^r$ cliques of size~$r$, which proves the desired estimate.

As we shall not need the converse direction, we only give a sketch of its proof. 
Let a weighted graph $\GG$ of order $n$ specified by the sequence $(x_1, x_2, \ldots, x_n)$ 
of~$n$ real numbers and by the function $a\colon [n]^{(2)}\lra [0, 1]$ be given. 
Consider a large integer $k$ and form a graph $H$ whose vertices fall into $n$ 
independent classes $V_1, V_2, \ldots, V_n$ whose sizes are approximately 
$kx_1, kx_2, \ldots, kx_n$ respectively, and in which for each unordered pair 
$\{i, j\}\in [n]^{(2)}$ roughly a proportion of $a_{ij}$ among all possible edges from 
$V_i$ to $V_j$ is present in a sufficiently random way. 
Such a graph $H$ can in particular be arranged to have $k$ vertices, 
${\GG(K_2)\cdot k^2\pm O(k)}$ edges and $\GG(K_r)\cdot k^r\pm O(k^{r-1})$ cliques of size $r$, 
so letting~$k$ tend to infinity we may in fact derive Claim~\ref{Cl21} from 
Conjecture~\ref{CDC}.

Throughout the rest of this article, we follow Nikiforov's suggestion~\cite{Nik} 
to think about the clique density problem in terms of weighted graphs.

\section{The convex lower bound}
\label{sec:convex}

In this section we discuss an analogue of the convex bound~\eqref{eq:conv} adapted to
the setting of weighted graphs. The proof we describe is essentially the same as that 
given by Kad{\v{z}}iivanov and~Nikiforov in~\cite{Had-Nik}. Nevertheless it might be
helpful to include full details here, because the second step of the proof of the 
clique density theorem to be given in Section~\ref{sec:cdt} will involve some similar 
calculations. 

Moreover, our proof of the clique density theorem uses the convex bound in two quite 
different ways. 
First, it implies that the clique density theorem holds for the critical values 
$\gamma=\tfrac{t}{2(t+1)}$, where $t$ is some positive integer, while our approach to 
Claim~\ref{Cl21} cannot deal with these values due to non-differentiability issues.
Second, it is going to be helpful later on to have some ``approximate version''
of the clique density theorem available.
 
\begin{prop} \label{pr31}
Given a weighted graph $\GG$ and an integer $r\gs 2$ such that the quantity 
$\gamma=\GG(K_2)$ is not smaller than $\tfrac{r-2}{2(r-1)}$, we have
\[
\GG(K_r)\gs \frac 1{r!}\cdot 2\gamma(4\gamma-1)(6\gamma-2)\cdot\ldots\cdot
\bigl(2(r-1)\gamma-(r-2)\bigr)\,.
\]
\end{prop}

\begin{proof}
Clearly this follows by means of an easy induction on $r$ from the following statement:

\smallskip

{\it \hskip.5cm $(*)$ \hskip.5cm If a weighted graph $\GG$ satisfies $\gamma=\GG(K_2)\gs \tfrac{r-2}{2(r-1)}$ 
for some $r\gs 2$, then
\[
\GG(K_r)\gs \frac{2(r-1)\gamma-(r-2)}r\cdot\GG(K_{r-1}) 
\quad \tn{ \it and} \quad 
\GG(K_{r-1})>0\,.
\]
}

Thus it suffices to verify $(*)$ instead and this will again be done by induction on $r$. 
The base case $r=2$ is obvious in view of $\GG(K_1)=1$. 
So suppose now that $\GG$ is a weighted graph satisfying 
$\gamma=\GG(K_2)\gs \tfrac{r-1}{2r}>\tfrac{r-2}{2(r-1)}$ for some $r\gs 2$ and that
\[
r\cdot \GG(K_r)\gs \bigl(2(r-1)\gamma-(r-2)\bigr)\cdot\GG(K_{r-1}) 
\quad \tn{as well as } \quad 
\GG(K_{r-1})>0
\]
hold.
For the induction step, we remark that these assumptions trivially entail $\GG(K_r)>0$,
so that it only remains to estimate $\GG(K_{r+1})$ from below. Let $n$ denote the order 
of $\GG$ and suppose that $\GG$ is given by the sequence $(x_1, x_2, \ldots, x_n)$ of $n$ 
reals numbers and by the function $a\colon [n]^{(2)}\lra [0, 1]$.

For each $M\subseteq [n]$ we write
\[
A_M=\prod_{e\in M^{(2)}}a_e 
\quad \tn{ and } \quad 
X_M=\prod_{i\in M}x_i\,.
\]
Now consider any $M\in [n]^{(r+1)}$ and define
\[
B_M=\sum_{e\in M^{(2)}}\prod_{f\in M^{(2)}-\{e\}} a_f
\]
as well as
\[
C_M=\sum_{N\in M^{(r)}} A_N\,.
\]

We claim that these expressions satisfy
\begin{equation} \label{eq:ABCM}
2B_M-C_M\ls (r^2-1)A_M\,.
\end{equation}
To see this, we note that this inequality is linear in each of its variables~$a_e$ with
$e\in M^{(2)}$, which entails that we only need to look at the case where 
$a_e\in\{0, 1\}$ holds for all $e\in M^{(2)}$. 
Now if additionally the number 
\[
K=\tn{{\tt \#}}\,\bigl\{e\in M^{(2)}\,\big|\, a_e=0\bigr\}
\]
is at least $2$, then $A_M=B_M=0$, and $C_M\gs 0$; if $K=1$, then $A_M=0$, $B_M=1$, $C_M=2$; 
and finally if $K=0$, then $A_M=1$, $B_M=\tfrac 12 r(r+1)$, and $C_M=r+1$. 
This completes the proof of~\eqref{eq:ABCM}.

Multiplying this estimate by $X_M$ and summing over all possibilities for~$M$, we infer
\begin{equation} \label{eq:32}
\sum\limits_{M\in [n]^{(r+1)}} (2B_M-C_M)X_M\ls (r^2-1)\GG(K_{r+1})\,.
\end{equation}
Setting
\[
\eta_L=\sum_{i\in [n]-L}x_i\prod_{\ell\in L}a_{i\ell}
\]
for all $L\in [n]^{(r-1)}$, we shall now investigate the sum
\[
\Omega=\sum_{L\in [n]^{(r-1)}}A_LX_L\eta_L^2\,.
\]
Expanding the squares, we get several ``quadratic terms'' in which some $x_i^2$ appears 
as a factor and some ``mixed terms'' for which this is not the case. 
Let~$\Omega_{\mathrm{sq}}$ and $\Omega_{\mathrm{mix}}$ denote the corresponding sums. 
In view of the inequality $a_e^2\ls a_e$, that is valid for all $e\in [n]^{(2)}$, we
may estimate
\begin{align*}
\Omega_{\mathrm{sq}} &\ls 
\sum_{L\in [n]^{(r-1)}}A_LX_L\sum_{i\in [n]-L}x_i^2\prod_{\ell\in L}a_{i\ell}
=\sum_{Q\in [n]^{(r)}}A_QX_Q\sum_{i\in Q}x_i \\
& =\sum_{Q\in [n]^{(r)}}A_QX_Q\Big(1-\sum_{i\in [n]-Q}x_i\Big)
=\GG(K_r)-\sum_{M\in [n]^{(r+1)}}C_MX_M\,.
\end{align*}
Moreover, we have
\[
\Omega_{\mathrm{mix}}=2\sum_{M\in [n]^{(r+1)}} B_MX_M\,,
\]
whence
\[
\Omega=\Omega_{\mathrm{sq}}+\Omega_{\mathrm{mix}}\ls 
\GG(K_r)+\sum\limits_{M\in [n]^{(r+1)}} (2B_M-C_M)X_M\,.
\]
In combination with~\eqref{eq:32} this yields
\[
\Omega\ls\GG(K_r)+(r^2-1)\GG(K_{r+1})\,.
\]
On the other hand, we get a lower bound on $\Omega$ from the Cauchy-Schwarz inequality,
\[
\Bigg(\sum_{L\in [n]^{(r-1)}}A_LX_L\eta_L\Bigg)^2\ls 
\sum_{L\in [n]^{(r-1)}}A_LX_L\cdot \sum_{L\in [n]^{(r-1)}}A_LX_L\eta_L^2 \,,
\]
where
\[
\sum_{L\in [n]^{(r-1)}}A_LX_L\eta_L=r\sum_{Q\in [n]^{(r)}}A_QX_Q=r\cdot\GG(K_r)
\]
and
\[
\sum_{L\in [n]^{(r-1)}}A_LX_L=\GG(K_{r-1})\,.
\]
So altogether we have
\[
r^2\GG(K_r)^2\ls \GG(K_{r-1})\left(\GG(K_r)+(r^2-1)\GG(K_{r+1})\right).
\]
Invoking now the induction hypothesis, we obtain, after a permissible 
cancelation of $\GG(K_{r-1})$, that
\[
\bigl(2(r-1)r\gamma-r(r-2)\bigr)\GG(K_r)\ls\GG(K_r)+(r^2-1)\GG(K_{r+1})\,,
\]
and hence indeed
\[
\bigl(2r\gamma-(r-1)\bigr)\GG(K_r)\ls (r+1)\GG(K_{r+1})\,,
\]
which completes the induction step. This finally proves $(*)$ and thus the proposition.
\end{proof}

We would now like to make those consequences of Proposition~\ref{pr31} explicit that 
we shall really utilise in the sequel. The following corollary is a slight modification 
of inequality~(25) from~\cite{Had-Nik}.

\begin{cor}\label{cor33}
Suppose that $r$ and $s$ are integers satisfying $r\gs 2$ and $s\gs r-1$.
Then for every weighted graph $\GG$ satisfying $\gamma=\GG(K_2)>\tfrac{s-1}{2s}$ one has
\[
\GG(K_r)> \frac 1s\cdot\binom{s}{r}\cdot \left(\frac{2\gamma}{s-1}\right)^{r-1}\,.
\]
\end{cor}

\begin{proof}
Clearly $\frac{s-1}{2s}\gs\tfrac{r-2}{2(r-1)}$, wherefore Proposition~\ref{pr31} tells us
\[
\GG(K_r)\gs \frac 1{r!}\cdot 2\gamma(4\gamma-1)(6\gamma-2)
\cdot\ldots\cdot\bigl(2(r-1)\gamma-(r-2)\bigr)\,.
\]
Now for each $i\in\{1, 2, \ldots, r-1\}$ we have
\[
2i\gamma-(i-1)> 2\gamma\left(i-\frac{s(i-1)}{s-1}\right)=\frac{2\gamma(s-i)}{s-1}\gs 0
\]
and hence
\[
\GG(K_r)> \left(\frac{2\gamma}{s-1}\right)^{r-1}\cdot \frac{(s-1)\cdot\ldots\cdot (s-r+1)}{r!}=\frac 1s\cdot\binom{s}{r}\cdot \left(\frac{2\gamma}{s-1}\right)^{r-1}\,. \qedhere
\]
\end{proof}

\begin{cor}\label{cor34}
Under the additional assumption $\alpha\in\left\{0, \tfrac 1s\right\}$, Claim~\ref{Cl21} holds.
\end{cor}

\begin{proof}
We have already seen in the introduction that we have $\gamma=\tfrac t{2(t+1)}$ 
for some nonnegative integer $t$ in these cases. If $t\ls r-2$ our claim is obvious, 
so we may suppose $t\gs r-1$ from now on. Now $\gamma$ is large enough for 
Proposition~\ref{pr31} to be applicable and the desired result follows.
\end{proof} 

\section{Some analytical preparations}
\label{sec:anna}

This section provides a thorough analysis of the function occurring in Claim~\ref{Cl21}. 
It also includes some further technical results that we will need in the next section 
for the proof of the clique density theorem.

Throughout the present section, we fix two integers $r\gs 3$ and $s\gs r-1$ as well as a 
real number $M\gs 1$ satisfying
\begin{equation} \label{eq:Msmall}
\left(\frac{s-1}s\right)^{r-2}>\frac{s-r+1}{s-1}\cdot M^{r-2}\,.
\end{equation}

Define the function $F_r\colon\left[0, \tfrac 12\right)\lra \left[0, \tfrac 1{r!}\right)$ 
as follows: given $\gamma\in\left[0, \tfrac 12\right)$, choose the unique positive integer 
$t$ for which $\gamma\in\bigl[\tfrac{t-1}{2t}, \tfrac{t}{2(t+1)}\bigr)$ is true, 
determine the real number $\alpha\in\left(0, \tfrac 1t\right]$ solving the equation 
$\gamma=\tfrac t{2(t+1)}(1-\alpha^2)$, and set
\[
F_r(\gamma)=\frac 1{(t+1)^r}\binom{t+1}{r}(1+\alpha)^{r-1}\bigl(1-(r-1)\alpha\bigr)\,.
\]
In particular, we have 
\begin{equation}\label{eq:Frt}
F_r\left(\frac{t-1}{2t}\right)=\frac{1}{t^r}\binom{t}{r}
\end{equation}
for every positive integer $t$.

In terms of this function, the statement of Claim~\ref{Cl21} can be shortened to the 
inequality $\GG(K_r)\gs F_r(\GG(K_2))$, that is allegedly valid for all weighted 
graphs~$\GG$. We have more or less already seen earlier that $F_r$ is continuous and 
clearly it is piecewise differentiable as well. Moreover, $F_r$ vanishes identically 
on the interval $\bigl[0, \tfrac{r-2}{2(r-1)}\bigr]$. If 
$\gamma\in\bigl(\tfrac{t-1}{2t}, \tfrac{t}{2(t+1)}\bigr)$ holds for some integer 
$t\gs r-1$, then differentiating the equation locally defining $F_r(\gamma)$ with 
respect to $\alpha$, we infer
\[
-\frac {t\alpha}{t+1}\cdot F_r'(\gamma)=
-\frac{(r-1)r}{(t+1)^r}\binom{t+1}{r}\alpha(1+\alpha)^{r-2}\,.
\]
As $\alpha>0$, it follows that
\begin{equation}\label{eq:F-deriv}
F_r'(\gamma)=\frac{(r-1)r}{t(t+1)^{r-1}}\binom{t+1}{r}(1+\alpha)^{r-2}>0\,.
\end{equation}
Thus $F_r$ is strictly increasing on the interval 
$\bigl[\tfrac{r-2}{2(r-1)}, \tfrac 12\bigr)$ and, as
\[
\lim_{\gamma\lra 1/2}F_r(\gamma)=\lim_{t\lra\infty} \frac 1{(t+1)^r}\binom{t+1}{r}
=\frac 1{r!}\,,
\]
it possesses an inverse 
\[
F_r^{-1}\colon\bigl[0, \tfrac 1{r!}\bigr)\lra \bigl[\tfrac{r-2}{2(r-1)}, \tfrac 12\bigr)\,.
\]
Moreover, the above expression for $F_r'(\gamma)$ decreases as $\alpha$ decreases, 
whence $F_r$ is in an obvious sense piecewise concave. 
Notice that the identity function~$F_2$ on $\left[0, \tfrac 12\right)$ has essentially 
the same properties as $F_r$. This concludes our discussion of the most elementary 
properties of these functions.

Next we propose to look at the function 
\[
H\colon\left[\tfrac{r-2}{r-1}\cdot M, M\right]\lra\mathbb{R}^+_0
\]
given by
\[
\eta\longmapsto \frac 1{s^{r-1}}\binom{s}{r-1}\cdot\frac{(r-1)\eta-(r-2)M}{\eta^{r-1}}\,.
\]
\begin{claim} \label{Cl41}
The function $H$ is strictly increasing and satisfies
\[
F_{r-1}\left(\frac{s-2}{2(s-1)}\right)<H(M)\ls F_{r-1}\left(\frac{s-1}{2s}\right)\,.
\]
\end{claim}

\begin{proof}
If $\eta\in \left[\tfrac{r-2}{r-1}\cdot M, M\right)$, then
\[
H'(\eta)=\frac {(r-2)(r-1)}{s^{r-1}}\binom{s}{r-1}\cdot\frac{M-\eta}{\eta^{r}}>0\,,
\]
which entails the first part of our claim. Furthermore, by~\eqref{eq:Msmall} we have
\[
H(M)=\frac 1{s^{r-1}}\binom{s}{r-1}\cdot\frac 1{M^{r-2}}> \frac{s-r+1}{s(s-1)^{r-1}}
\binom{s}{r-1}=\frac 1{(s-1)^{r-1}}\binom{s-1}{r-1}\,,
\]
i.e.,
\[
H(M)>F_{r-1}\left(\frac{s-2}{2(s-1)}\right)\,.
\]
Finally, using $M\gs 1$, we get
\[
H(M)=\frac 1{s^{r-1}}\binom{s}{r-1}\cdot\frac 1{M^{r-2}}\ls 
\frac 1{s^{r-1}}\binom{s}{r-1}=F_{r-1}\left(\frac{s-1}{2s}\right)\,. \qedhere
\]
\end{proof}

Notice in particular that the composition $F_{r-1}^{-1}\circ H$ is defined everywhere 
on the interval $\left[\tfrac{r-2}{r-1}\cdot M, M\right]$ and that its range is contained 
in $\bigl[\tfrac{r-3}{2(r-2)}, \tfrac{s-1}{2s}\bigr]$, which in turn is included in 
$\left[0, \tfrac 12\right)$. For $t\in\{r-2, r-1, \ldots, s-1\}$ there exists a unique 
real number $\vt_t\in \left[\tfrac{r-2}{r-1}\cdot M, M\right]$ satisfying 
$H(\vt_t)=F_{r-1}\left(\tfrac{t-1}{2t}\right)$. The number $\vt_{s-1}$ will play a 
special r\^{o}le later and sometimes it will just be denoted by $\vt$. Evidently one has
\[
\tfrac{r-2}{r-1}\cdot M=\vt_{r-2}<\vt_{r-1}<\ldots<\vt_{s-1}=\vt< M\,.
\]
Furthermore~\eqref{eq:Frt} leads to
\begin{equation}\label{eq:Htt}
H(\vt_t)=\frac{1}{t^{r-1}}\binom{t}{r-1}
\end{equation}
for any $t\in\{r-2, r-1, \ldots, s-1\}$. In particular, for $t=s-1$ we get
\begin{equation}\label{eq:Htheta}
\vt^{r-1}=\left(\frac{s-1}{s}\right)^{r-1}
\cdot\frac{s}{s-r+1}\cdot\bigl((r-1)\vt-(r-2)M\bigr)
\end{equation}
due to the definition of~$H$. 

Later on we shall need some estimates concerning these numbers $\vt_t$.

\begin{claim} \label{Cl42}
If the integer $t$ belongs to the interval $[r-2, s-2]$, 
then $\vt_t\ls\tfrac{t}{t+1}\cdot M$. In addition, we have $\vt\gs \tfrac{s-1}s\cdot M$.
\end{claim}

\begin{proof}
Whenever $t\in[r-2, s-2]$ is an integer, we have $H(\vt_{t+1})\ls H(M)$. Owing 
to~\eqref{eq:Htt} and the definiton of $H$ this yields
\[
\frac 1{(t+1)^{r-1}}\binom{t+1}{r-1}\ls \frac 1{s^{r-1}}\binom{s}{r-1}\frac 1{M^{r-2}}\,.
\]
Multiplying this by $(t-r+2)(t+1)^{r-2}/t^{r-1}$ we infer 
$H(\vt_t)\ls H(\tfrac t{t+1}\cdot M)$, thus proving the first part of our claim. 
Similarly but slightly easier we deduce from $M\gs 1$ that 
$H(\tfrac{s-1}s\cdot M)\ls H(\vt)$, which leads to the second part of the claim.
\end{proof}

\begin{claim} \label{Cl43}
If $\eta\in \left[\tfrac{r-2}{r-1}\cdot M, M\right]$ and $\nu=(F_{r-1}^{-1}\circ H)(\eta)$, 
then the function 
\[
Q\colon [0, \nu]\lra\mathbb{R}
\]
defined by
\[
\delta\longmapsto (r-1)\binom{s}{r-1}\delta-s^{r-1}\eta^{r-2}F_r(\delta)
\]
attains its global maximum at $\delta=\nu$.
\end{claim}

\begin{proof}
Choose an integer $t\in[r-2, s-1]$ as well as a real number 
$\beta\in\bigl[0, \tfrac 1t\bigr]$ such that $\nu=\tfrac t{2(t+1)}(1-\beta^2)$. 
Since $Q$ is piecewise convex and convex functions attain their global maxima at 
boundary values, it suffices to establish the following two statements:
\begin{enumerate}
\item[$(A)$] The function $Q$ is increasing on $\left[\tfrac{t-1}{2t}, \nu\right]$.
\item[$(B)$] If $d\in[t-1]$, then $Q(\tfrac{d-1}{2d})\ls Q(\tfrac{t-1}{2t})$.
\end{enumerate}
For the proofs of both of these subclaims, we use
\begin{enumerate}
\item[$(C)$] $M^{r-2}\cdot \frac 1{t^{r-1}}\binom{t}{r-1}\ls\frac 1{s^{r-1}}\binom{s}{r-1}$,
\end{enumerate}
which is an obvious consequence of $H(\vt_t)\ls H(M)$. 

Now, to verify $(A)$, take any $\delta\in\left(\tfrac{t-1}{2t}, \nu\right)$ and write 
$\delta=\tfrac t{2(t+1)}(1-\alpha^2)$, where $\alpha\in\left(\beta, \tfrac 1t\right)$. 
Multiplying $(C)$ by $(r-1)s^{r-1}$ one obtains
\[
s^{r-1}M^{r-2}\frac{(r-1)r}{t(t+1)^{r-1}}\binom{t+1}{r}\left(\frac{t+1}t\right)^{r-2}
\ls (r-1)
\binom{s}{r-1}\,.
\]
In view of $\eta\ls M$ and $\alpha\ls \frac 1t$ this implies
\[
s^{r-1}\eta^{r-2}F_r'(\delta)\ls (r-1)\binom{s}{r-1}\,,
\]
whence $Q'(\delta)\gs 0$. Thereby we have proved assertion $(A)$.

Let us now turn our attention to $(B)$. If $t\ls r-1$, then $F_r$ vanishes at all 
relevant numbers and our claim is obvious. So henceforth we may suppose $t\gs r$ and for 
similar reasons $d\gs r-1$ as well. The function 
\[
\Phi\colon\left[\tfrac 1t, \tfrac 1{r-1}\right]\lra[0, 1]
\]
defined by
\[
x\longmapsto (1-x)(1-2x)\cdot\ldots\cdot \bigl(1-(r-1)x\bigr)
\]
is obviously convex, wherefore
\[
\frac{\Phi\left(\tfrac 1t\right)-\Phi\left(\tfrac 1d\right)}{\tfrac 1d-\tfrac 1t}
\ls -\Phi'\left(\tfrac 1t\right)\,.
\]
Since
\[
-\Phi'\left(\tfrac 1t\right)
=\left\{\frac{t}{t-1}+\frac{2t}{t-2}+\ldots+\frac{(r-1)t}{t-r+1}\right\}
\Phi\left(\tfrac 1t\right)
\ls\frac{(r-1)rt}{2(t-r+1)}\Phi\left(\tfrac 1t\right)\,,
\]
it follows that
\begin{align*}
\frac 1{t^r}\binom{t}{r}- \frac 1{d^r}\binom{d}{r}&
\ls\left(\frac 1d-\frac 1t\right)\frac{(r-1)r}{2(t-r+1)}\cdot\frac 1{t^{r-1}}\binom{t}{r} \\
&=(r-1)\left(\frac{t-1}{2t}-\frac{d-1}{2d}\right)\cdot \frac 1{t^{r-1}}\binom{t}{r-1}\,.
\end{align*}
Multiplying this by
\[
s^{r-1}\eta^{r-2}\cdot \frac 1{t^{r-1}}\binom{t}{r-1}\ls \binom{s}{r-1}\,,
\]
which in view of $\eta\ls M$ is a consequence of $(C)$, we deduce
\[
s^{r-1}\eta^{r-2}\left\{\frac 1{t^r}\binom{t}{r}- \frac 1{d^r}\binom{d}{r}\right\}
\ls (r-1)\binom{s}{r-1}\left(\frac{t-1}{2t}-\frac{d-1}{2d}\right),
\]
which is easily seen to be equivalent to $(B)$.
\end{proof}

The following result will only be used for $k=2$ and $k=r$, 
but the general case is not really harder.

\begin{claim} \label{Cl44}
For each integer $k\gs 2$ the function 
\[
J_k\colon \bigl[\tfrac{r-2}{r-1}\cdot M, M\bigr]\lra\mathbb{R}
\]
defined by
\[
\eta\longmapsto\eta^k(F_k\circ F_{r-1}^{-1}\circ H)(\eta)
\]
is concave or convex on the intervals $\left[\vt_t, \vt_{t+1}\right]$, 
where $t\in\{r-2, r-1,\ldots, s-2\}$, and $\left[\vt, M\right]$, depending on whether 
$k\gs r-1$ or $k\ls r-1$.
\end{claim}
\begin{proof}
Treating both cases for $k$ at the same time, we select any $t\in\{r-2, r-1, \ldots, s-1\}$
and intend to verify that the second derivative of $J_k$ has the expected sign on 
$\left(\vt_t, \vt_{t+1}\right)$, where for convenience $\vt_s=M$. Utilising that 
$x\longmapsto\tfrac{2x-1}{x^2}$ is strictly increasing on $(0, 1)$ 
we may define a function 
\[
S\colon\left(\vt_t, \vt_{t+1}\right)\lra\left(\tfrac{t}{t+1}, 1\right)
\]
such that
\[
(F_{r-1}^{-1}\circ H)(\eta)=\frac t{2(t+1)}\cdot\frac{2S(\eta)-1}{S(\eta)^2}
\]
holds for all $\eta\in\left(\vt_t, \vt_{t+1}\right)$. 
Since the right hand side may be rewritten as
\[
\frac t{2(t+1)}\times\left\{1-\left(\frac 1{S(\eta)}-1\right)^2\right\}\,,
\]
we have
\[
\frac 1{s^{r-1}}\binom{s}{r-1}\frac{(r-1)\eta-(r-2)M}{\eta^{r-1}}
=\frac 1{(t+1)^{r-1}}\binom{t+1}{r-1}\frac{(r-1)S(\eta)-(r-2)}{S(\eta)^{r-1}}.
\]
Differentiating with respect to $\eta$ and dividing by $(r-2)(r-1)$, we find
\[
\frac 1{s^{r-1}}\binom{s}{r-1}\cdot\frac{M-\eta}{\eta^{r}}
=\frac 1{(t+1)^{r-1}}\binom{t+1}{r-1}\cdot\frac{1-S(\eta)}{S(\eta)^{r}}\cdot S'(\eta)\,,
\]
and the combination of both equations yields
\[
S'(\eta)=\frac{S(\eta)(M-\eta)\left[(r-1)S(\eta)-(r-2)\right]}
{\eta(1-S(\eta))\left[(r-1)\eta-(r-2)M\right]}\,.
\]
Furthermore
\[
J_k(\eta)=\frac 1{(t+1)^{k}}\binom{t+1}{k}\cdot\frac{kS(\eta)-(k-1)}{S(\eta)^k}\cdot\eta^k\,.
\]
Differentiating and using the above formula for $S'(\eta)$, we get
\[
J'_k(\eta)=\frac 1{(t+1)^{k}}\binom{t+1}{k}
\frac{k\eta^{k-1}\left[(r-1)S(\eta)\eta-(k-1)\eta+(k-r+1)S(\eta)M\right]}
{S(\eta)^k\left[(r-1)\eta-(r-2)M\right]}\,.
\]

A repetition of this argument leads to
\[
J''_k(\eta)=\frac 1{(t+1)^{k}}\binom{t+1}{k}\cdot
\frac{k(k-1)(r-k-1)\eta^{k-2}\left[S(\eta)M-\eta\right]^2}
{S(\eta)^k\left(1-S(\eta)\right)\left[(r-1)\eta-(r-2)M\right]^2}\,,
\]
which entails the desired conclusion in view of the presence of the factor $r-k-1$ in the numerator.
\end{proof}

\begin{claim} \label{Cl45}
For each $\eta\in\left[\frac{r-2}{r-1}\cdot M, M\right]$ the difference
\[
(r-1)\binom{s}{r-1}\eta^2\nu -s^{r-1}\eta^rF_r(\nu)
\]
is at most
\[
\frac{r-2}{(s-1)(s+1)}\binom{s+1}{r}\cdot
\left(\tfrac 12(r-1)s\vt^2-(r-1)s\vt M+r(s-1)M\eta\right)\,,
\]
where $\nu=(F_{r-1}^{-1}\circ H)(\eta)$.
\end{claim}

\begin{proof}
The difference under consideration, which we shall denote by $T(\eta)$ in the sequel, 
rewrites in the notation of Claim~\ref{Cl44} as
\[
T(\eta)=(r-1)\binom{s}{r-1}J_2(\eta)-s^{r-1}J_k(\eta)\,.
\]
In the special case $\eta=\vt$ we have $\nu=\tfrac{s-2}{2(s-1)}$ due to the definition 
of $\eta$ and using~\eqref{eq:Htheta} and~\eqref{eq:Frt} it is not hard to check that 
we have equality in the inequality we seek to establish. 
Moreover, Claim~\ref{Cl44} shows that $T$ is piecewise convex as a function of $\eta$.

For these reasons it suffices to prove the statements

\begin{enumerate}
\item[$(A)$] If $r-2\ls t\ls s-2$, then 
	$\lim\limits_{\eta\lra\vt_t^+}T'(\eta)\gs (r-2)\binom{s}{r-1}M$
\item[$(B)$] and $\lim\limits_{\eta\lra M^-}T'(\eta)\ls (r-2)\binom{s}{r-1}M$,
\end{enumerate}
where the superscripted plus or minus signs below the limit are intended to signify 
that~$\eta$ is supposed to approach the boundary value in question from the right or from the 
left, respectively. Throughout the computations that follow we use the function $S$ 
as well as the formulae for $J'(\eta)$ corresponding to $k=2$ and $k=r$ from the foregoing 
proof.

To verify $(A)$, we distinguish two cases.

\smallskip

{\it First Case: $t=r-2$}

\noindent Note that the hypothesis of $(A)$ yields $s\gs r$, wherefore $H(\vt_{r-1})\ls H(M)$, 
i.e.,
\[
\frac 1{(r-1)^{r-1}}\ls \frac 1{s^{r-1}}\binom{s}{r-1}\cdot\frac 1{M^{r-2}}\,.
\]
Now let $\eta\in\left(\vt_{r-2}, \vt_{r-1}\right)$ be arbitrary. Then
\[
\frac{(r-1)S(\eta)-(r-2)}{(r-1)^{r-1}\cdot S(\eta)^{r-1}}
\ls \frac 1{s^{r-1}}\binom{s}{r-1}
\frac{(r-1)S(\eta)-(r-2)}{S(\eta)^{r-1}\cdot M^{r-2}}\,,
\]
which by the definition of $S$ and $H$ yields $H(\eta)\ls H(S(\eta)M)$, 
and thus ${\eta\ls S(\eta)M}$. Also, $\eta\in\left(\tfrac{r-2}{r-1}\cdot M, M\right)$ gives
\[
\eta\gs (r-1)\eta-(r-2)M>0\,,
\]
which by $S(\eta)\ls 1$ may be weakened to
\[
\eta\gs S(\eta)\bigl((r-1)\eta-(r-2)M\bigr)\,.
\]
Consequently
\[
\eta\bigl(S(\eta)M-\eta\bigr)\gs 
S(\eta)\bigl(S(\eta)M-\eta\bigr)\bigl((r-1)\eta-(r-2)M\bigr)\,,
\]
i.e.,
\[
\eta\left[(r-1)S(\eta)\eta-\eta-(r-3)S(\eta)M\right]
\gs MS(\eta)^2\left[(r-1)\eta-(r-2)M\right]\,.
\]
Multiplying this by $(r-2)\binom{s}{r-1}S(\eta)^{-2}\left[(r-1)\eta-(r-2)M\right]^{-1}$ 
we infer
\[
T'(\eta)\gs (r-2)\binom{s}{r-1}M\,,
\]
as required.

\smallskip

{\it Second Case: $r-1\ls t\ls s-2$.}

\noindent Notice that Claim~\ref{Cl42} entails
\[
(r-2)(M-\vt_t)(M-\tfrac{t+1}t\vt_t)\gs 0\,,
\]
whence
\[
\vt_t\left[(r-2-\tfrac 1t)\vt_t-(r-3)M\right]
\gs \bigl(M-\tfrac 1t\vt_t\bigr)\bigl[(r-1)\vt_t-(r-2)M\bigr]\,.
\]
By $t>r-2$ the second factor on the right-hand side is positive, wherefore
\[
\vt_t\cdot\frac{(r-2-\tfrac 1t)\vt_t-(r-3)M}{(r-1)\vt_t-(r-2)M}\gs M-\tfrac 1t\vt_t\,.
\]
Multiplying by $(r-1)$ and subtracting $M-\tfrac{r-1}t\vt_t$ we obtain
\[
(r-1)\vt_t\cdot
\frac{\bigl(r-2-\tfrac 1t\bigr)\vt_t-(r-3)M}{(r-1)\vt_t-(r-2)M}
-\bigl(M-\tfrac{r-1}t\vt_t\bigr)
\gs (r-2)M\,.
\]
If we now multiply by $\binom{s}{r-1}$ and use
\[
\lim_{\eta\lra\vt_t^+}S(\eta)=\frac{t}{t+1}
\]
as well as the definition of $\vt_t$, we get indeed
\[
\lim\limits_{\eta\lra\vt_t^+}T'(\eta)\gs (r-2)\binom{s}{r-1}M.
\]
This completes the verification of $(A)$.

So let us now continue with $(B)$. From $H(M)>F_{r-1}\left(\tfrac{s-2}{2(s-1)}\right)$ 
one deduces easily
\[
S(M)>\frac{s-1}s\gs \frac{r-2}{r-1}\,,
\]
where we have made the obvious definition
\[
S(M)=\lim_{\eta\lra M^-}S(\eta)\,.
\]
This entails
\[
\frac{(s-r+1)S(M)}{(r-1)S(M)-(r-2)}\ls s-1\,,
\]
which in turn implies
\[
(r-1)\left\{\frac{(s-r+1)S(M)}{(r-1)S(M)-(r-2)}-(s-1)\right\}\left(\frac{1-S(M)}{S(M)}\right)^2
\ls 0\,.
\]
Adding $(r-2)s$ and rearranging our terms, we infer
\[
(r-1)(s-1)\frac{2S(M)-1}{S(M)^2}
-\frac{(s-r+1)\left(rS(M)-(r-1)\right)}{S(M)\left((r-1)S(M)-(r-2)\right)}\ls (r-2)s\,.
\]
Multiplying this by $\tfrac Ms\cdot\binom{s}{r-1}$ and exploiting the equation
\[
\frac 1{M^{r-2}}=\frac{(r-1)S(M)-(r-2)}{S(M)^{r-1}}\,,
\]
that follows easily from the definition of $S$, one gets assertion $(B)$, 
whereby Claim~\ref{Cl45} has finally been proved.
\end{proof} 

\section{Clique densities}
\label{sec:cdt}

We now come to the central section of this article, in which we are going to provide a 
proof of Claim~\ref{Cl21}, thus solving the clique density problem.

\begin{thm} \label{Thm51}
If $\GG$ is a weighted graph and $r\gs 2$ is an integer, then we have 
\[
	\GG(K_r)\gs F_r(\GG(K_2))\,.
\]
In other words, the estimate
\[
\GG(K_r)\gs 
\frac 1{(s+1)^r}\cdot\binom{s+1}{r}\cdot (1+\alpha)^{r-1}\bigl(1-(r-1)\alpha\bigr)
\]
holds, where $s$ refers to a positive integer for which 
$\gamma=\GG(K_2)$ belongs to the interval $\bigl[\tfrac{s-1}{2s}, \tfrac{s}{2(s+1)}\bigr]$ 
and $\alpha\in\bigl[0, \tfrac 1s\bigr]$ is required to satisfy 
$\gamma=\tfrac s{2(s+1)}(1-\alpha^2)$.
\end{thm}

Before we proceed to the proof itself, we give an informal outline of the strategy we use: 
the proof will be by induction on $r$ and 
for fixed $r$ by induction on the order $n$ of $\GG$. Once~$r$ and $n$ are fixed we will, 
rather than considering all possible $\GG$, restrict our attention to a hypothetical 
``extremal'' counterexample, where ``extremal'' means that the difference
$F_r(\GG(K_2))-\GG(K_r)$ is as large as possible. 
Then we already know from Corollary~\ref{cor34} that $\alpha\ne 0$ and $\alpha\ne\tfrac 1s$.
From now on, $r$, $s$, and $\alpha$ will be the main ``global variables'' describing~$\GG$ 
so that we regard a quantity appearing in the proof as being ``known'' when we can express 
it in terms of them. For instance, $\gamma$, $F_r(\gamma)$, and the derivative 
$\lambda=F'_r(\gamma)$ that will appear shortly are ``known''. 

Now what remains to be done is to establish an inequality in $n+\binom{n}{2}$ variables, 
where $n$ of the variables, $x_1, \ldots, x_n$, are the ``weights of the vertices'' of $\GG$
and the remaining ones, denoted by $a_{ij}$, are the ``weights of the edges'' of~$\GG$. 
So Lagrange's theorem on multivariate functions is applicable and yields some information 
per variable. 

For the ``vertex with number $i$'' we get essentially that the density $\GG_i(K_{r-1})$ of 
$r$-cliques it belongs to depends linearly on its weighted degree $\GG_i(K_{1})$. 
These terms are defined precisely below. For now it may suffice to say that the slope of 
this linear function is the known number $\lambda=F'_r(\gamma)$ mentioned above, 
while its constant term is an unknown new parameter $\mu$. 

Similarly we could get an equation for each ``edge variable'' $a_{ij}\in(0,1)$ but if $a_{ij}\in\{0,1\}$ we just get an inequality because we can vary $a_{ij}$ in one direction only. The precise estimate we thus obtain is called~\eqref{eq:lag-e} below. 

There is actually a quite interesting argument due to Nikiforov~\cite{Nik}, that
would allow us to assume $a_{ij}\in\{0,1\}$ for all $i\ne j$ here.

To gain additional information on the unknown number $\mu$, we multiply the equation gotten 
for the $i^{\mathrm{th}}$ vertex above by the weighted degree of that vertex and sum over $i$. 
The equation thus obtained may then be weakened by a calculation inspired by Razborov's 
work \cites{RazF, RazT} to the main inequality
\[
(r-1)\GG(K_r)+(r+1)\GG(K_{r+1})\ls\lambda\bigl(\gamma+3\GG(K_3)\bigr)-2\gamma\mu\,.
\]

What is a bit confusing about this estimate is that when we want to use it for learning 
something about $r$-cliques, it seems as if we should already know something about triangles 
and $(r+1)$-cliques. The way out of this difficulty is that in the graph case these are just 
edges and $r$-cliques in the neighbourhoods of vertices. This idea works in our weighted 
setting as well and thus allows to bring the induction hypothesis applied to ``weighted 
neighbourhoods'' into the argument. The inequality that results relates the $x_i$ and 
$a_{ij}$ in a quite involved way but by the results of Section~\ref{sec:anna} we can make 
it more well-behaved with respect to them at the cost of introducing a further constant $\vt$ 
that depends on the unknown number $\mu$. This leads to an inequality that besides
$r$,~$s$, and~$\alpha$ involves only $\mu$ and $\vt$, but the original variables $x_i$ 
and $a_{ij}$ will be gone. Finally it will turn out that this 
inequality is wrong.
 
Now we proceed with the details.

\begin{proof}[Proof of Theorem~\ref{Thm51}]
Since $F_2$ is the identity function confined to $\left[0, \tfrac 12\right)$, the result 
is clear for $r=2$. Arguing indirectly, let $r\gs 3$ denote the least integer for which 
the clique density theorem can fail\footnote{By the results of Razborov and Nikiforov
(\cite{RazT}, \cite{Nik}) we could assume $r\gs 5$ here, but actually there is no need 
for doing so.} and take $n$ to be the least order that counterexamples can possibly have. 
As we have already mentioned earlier, the function
\[ 
F_{r}\colon \left[0, \tfrac 12\right)\lra \left[0, \tfrac 1{r!}\right)
\]
is continuous. 
Now the collection of all weighted graphs of order~$n$ may in an obvious fashion be 
regarded as a compact topological space, which implies that the continuous function 
defined on it by
\[
\GG\longmapsto \GG(K_r)-F_r(\GG(K_2))
\]
attains an absolute minimum. Now fix a weighted graph $\GG$ of order $n$ for which this 
minimal value occurs. This property of $\GG$ will be referred to as ``extremality''. 
Choose an integer $s\gs 1$ as well as a real number $\alpha\in\left[0, \tfrac 1s\right]$ 
such that the number $\gamma=\GG(K_2)$ can be written as 
\begin{equation}\label{eq:gamma}
\gamma=\tfrac s{2(s+1)}(1-\alpha^2)\,.
\end{equation}
 
By the hypothesized failure of our theorem, we have
\[
\GG(K_r)< 
\frac 1{(s+1)^r}\cdot\binom{s+1}{r}\cdot (1+\alpha)^{r-1}\left(1-(r-1)\alpha\right)\,,
\]
which clearly can only happen if $s\gs r-1$. Also, Corollary~\ref{cor34} tells us that 
$\alpha\in\left(0, \tfrac 1s\right)$, and so the function $F_r$ is differentiable at 
$\gamma$. As we have seen in~\eqref{eq:F-deriv}, its derivative $\lambda=F_r'(\gamma)$ 
is given by
\begin{equation}\label{eq:lambda}
\lambda=\frac{(r-1)r}{s(s+1)^{r-1}}\binom{s+1}{r}(1+\alpha)^{r-2}\,.
\end{equation}
Let $\GG$ as usual be presented by the sequence $(x_1, x_2, \ldots, x_n)$ of nonnegative 
reals summing up to~$1$ and by the function $a\colon [n]^{(2)}\lra [0, 1]$. 
The remainder of the proof proceeds in five steps. 

\medskip

{\it First Step: Exploiting extremality.}

\noindent Clearly each of the numbers $x_1, x_2, \ldots, x_n$ has to be positive, for if 
one of them would vanish we could simply omit it, thus obtaining another counterexample
whose order would be smaller than $n$. Given a sequence $i_1, i_2, \ldots, i_m$ of distinct 
integers from $[n]$ as well as another integer $\rho\gs 1$, we set
\[
\GG_{i_1, i_2,\ldots, i_m}(K_\rho)
=\sum_{M\in I^{(\rho)}}\prod_{(k, j)\in [m]\times M}a_{i_kj}\prod_{e\in M^{(2)}}a_e
\prod_{j\in M}x_j\,,
\]
where $I=[n]-\{i_1, i_2,\ldots, i_m\}$. Note that for $m=0$ this coincides with our 
earlier notation, so no confusion can arise. 

As an example, we mention that for every
$i\in [n]$ we have
\[
\GG_i(K_1)=\sum_{j\in [n]-\{i\}}a_{ij}x_j\,,
\]
which may be thought of as the weighted degree of the ``vertex with index $i$''.
These weighted degrees satisfy
\[
\sum_{i\in [n]}x_i\GG_i(K_1)=2\sum_{\{i,j\}\in[n]^{(2)}}x_ix_ja_{ij}=2\GG(K_2)\,,
\]
whence
\begin{equation}\label{eq:deg-sum}
\sum_{i\in [n]}x_i\GG_i(K_1)=2\gamma\,.
\end{equation}

Similarly $\GG_i(K_{r-1})$ may be regarded as the relative density 
of $r$-cliques containing ``the vertex $i$''. It is soon going to be 
important for us that this quantity is also the partial derivative of
$\GG(K_r)$ with respect to $x_i$.

Since $(x_1, x_2, \ldots, x_n)$ is an 
interior point of the simplex
\[
\bigl\{(\xi_1, \xi_2, \ldots, \xi_n)\in[0, 1]^n\,\big|\,\xi_1+\xi_2+\ldots+\xi_n=1\bigr\}\,,
\]
Lagrange's theorem concerning the extremal values of multivariate functions reveals, 
that due to the extremality of $\GG$ there exists a certain real constant $\mu$ such that
\begin{equation}\label{eq:lag-v}
\GG_i(K_{r-1})=\lambda\GG_i(K_1)-\mu
\end{equation}
holds for all $i\in[n]$. 

As we are now going to see, the extremality of $\GG$ also implies:
\begin{equation}\label{eq:lag-e}
\text{For each } 
\{i, j\}\in [n]^{(2)} 
\text{ one has } 
a_{ij}\bigl(\lambda-\GG_{ij}(K_{r-2})\bigr)\gs 0\,.
\end{equation}

This is obvious whenever $a_{ij}$ vanishes, so let us suppose now that this number is 
positive. If~$\eta$ denotes any sufficiently small positive real number, we may construct 
a weighted graph $\GG^\eta$ agreeing entirely with $\GG$ except for having 
$a^\eta_{ij}=a_{ij}-\eta$. Clearly one has $\GG^\eta(K_2)=\gamma-\eta x_ix_j$ and 
${\GG^\eta(K_r)=\GG(K_r)-\eta x_ix_j\GG_{ij}(K_{r-2})}$, wherefore
\[
\GG^\eta(K_r)-F_r(\GG^\eta(K_2))=\GG(K_r)-F_r(\gamma)
+\eta x_ix_j\bigl(\lambda-\GG_{ij}(K_{r-2})\bigr)\pm O(\eta^2)\,,
\]
which in view of the assumed extremality of $\GG$ entails $\lambda\gs\GG_{ij}(K_{r-2})$ 
and hence~\eqref{eq:lag-e}.

\medskip

{\it Second Step: An estimate about triangles, $r$-cliques, and $(r+1)$-cliques.}

\noindent As in the proof of Proposition~\ref{pr31}, we set
\[
A_M=\prod_{e\in M^{(2)}}a_e 
\quad \tn{ and } \quad 
X_M=\prod_{i\in M}x_i
\]
for all $M\subseteq [n]$. This time, however, we need the further stipulations
\[
B_M=\sum_{i\in M}\Big(\sum_{j\in M-\{i\}}a_{ij}\Big)A_{M-\{i\}}\,,
\]
\[
C_M=\sum_{Q\in M^{(r)}}A_Q\,,
\]
and
\[
D_M=\sum_{i\in M}\sum_{\{j, k\}\in(M-\{i\})^{(2)}}(1-a_{ij})(1-a_{ik})A_{M-\{i\}}
\]
for all $M\in[n]^{(r+1)}$. What we shall need to know about these expressions is:
\begin{equation}\label{eq:ABCD}
\text{If }
M\in[n]^{(r+1)}, 
\text{ then }
B_M-(r-1)C_M+D_M\gs(r+1)A_M\,.
\end{equation}

To see this, we note again that our inequality is linear in each of its variables, 
for which reason we may suppose $a_e\in\{0, 1\}$ for all $e\in M^{(2)}$. 
Form a graph $H$ with vertex set $M$ by putting an edge between $i, j\in M$ exactly if 
$a_{ij}=1$. If $H$ is free from cliques of size $r$, then $A_M=B_M=C_M=D_M=0$. 
If $H$ contains a unique such clique and $i$ further edges, where $0\ls i\ls r-2$, 
then $A_M=0$, $B_M=i$, $C_M=1$, and $D_M=\binom{r-i}{2}$, wherefore indeed 
$B_M-(r-1)C_M+D_M=\binom{r-i-1}{2}\gs (r+1)A_M$. If the graph $H$ possesses exactly two 
cliques of size~$r$, then it misses precisely one edge and $A_M=0$, $B_M=2(r-1)$, $C_M=2$ 
as well as $D_M=0$. Finally, if $H$ happens to be a clique, then $A_M=1$, $B_M=r(r+1)$, 
$C_M=r+1$ and $D_M=0$. This analysis proves~\eqref{eq:ABCD} in all possible cases.

Multiplying the inequality just obtained by $X_M$ and summing over $M$, we deduce
\begin{equation}\label{eq:54}
\sum\limits_{M\in [n]^{(r+1)}}\bigl(B_M-(r-1)C_M+D_M\bigr)X_M\gs (r+1)\GG(K_{r+1})\,.
\end{equation}

Next, we consider the sum
\[
\Omega=\sum_{i\in[n]}x_i\GG_i(K_1)\GG_i(K_{r-1})\,.
\]
Expanding the product, we get several terms involving one of the variables 
$x_1, x_2, \ldots, x_n$ quadratically and certain other ``linear'' terms. We are thus led
to a decomposition
\[
\Omega=\Omega_{\mathrm{sq}}+\Omega_{\mathrm{mix}}\,,
\]
where
\[
\Omega_{\mathrm{sq}}
=\sum_{\{i, j\}\in [n]^{(2)}}(x_i^2x_j+x_ix_j^2)a_{ij}^2\GG_{ij}(K_{r-2})
\]
and
\[
\Omega_{\mathrm{mix}}=\sum_{M\in[n]^{(r+1)}}B_MX_M\,.
\]
Owing to~\eqref{eq:lag-e} we have
\begin{equation}\label{eq:omegasq}
\Omega_{\mathrm{sq}}
\gs\sum_{\{i, j\}\in [n]^{(2)}}(x_i^2x_j+x_ix_j^2)a_{ij}\GG_{ij}(K_{r-2})-\lambda\Phi\,,
\end{equation}
where 
\[
\Phi=\sum_{\{i, j\}\in [n]^{(2)}}(x_i^2x_j+x_ix_j^2)(a_{ij}-a_{ij}^2)\,.
\]
The sum on the right-hand side of~\eqref{eq:omegasq} rewrites as
\[
(r-1)\sum_{Q\in[n]^{(r)}}A_QX_Q\Big(1-\sum_{q\in [n]-Q}x_q\Big)
=(r-1)\GG(K_r)-(r-1)\sum_{M\in[n]^{(r+1)}}C_MX_M\,.
\]
So altogether he have
\[
\Omega\gs(r-1)\GG(K_r)+\sum_{M\in[n]^{(r+1)}}\bigl(B_M-(r-1)C_M\bigr)X_M-\lambda\Phi\,,
\]
which due to~\eqref{eq:54} may be weakened to
\[
\Omega\gs(r-1)\GG(K_r)+(r+1)\GG(K_{r+1})-\lambda\Phi-\sum_{M\in[n]^{(r+1)}}D_MX_M\,.
\]
Now notice that trivially
\[
\sum_{M\in[n]^{(r+1)}}D_MX_M\ls
\sum_{i\in[n]}\sum_{\{j, k\}\in ([n]-\{i\})^{(2)}}
a_{jk}(1-a_{ij})(1-a_{ik})x_ix_jx_k\GG_{jk}(K_{r-2})\,.
\]
Hence, writing
\[
V_{\{i, j, k\}}=a_{jk}(1-a_{ij})(1-a_{ik})+a_{ik}(1-a_{ij})(1-a_{jk})+a_{ij}(1-a_{ik})(1-a_{jk})
\]
for all $\{i, j, k\}\in[n]^{(3)}$ and applying~\eqref{eq:lag-e} again, we obtain
\begin{equation}\label{eq:omegalow}
(r-1)\GG(K_r)+(r+1)\GG(K_{r+1})-\lambda\Phi-\lambda\sum_{T\in [n]^{(3)}}V_TX_T\ls\Omega\,.
\end{equation}

On the other hand~\eqref{eq:lag-v} and~\eqref{eq:deg-sum} allow us to write
\[ 
\Omega=\lambda\sum_{i\in[n]}x_i\GG_i(K_1)^2-2\gamma\mu\,.
\]
So in view of the calculation
\begin{align*}
\sum_{i\in[n]}x_i\GG_i(K_1)^2&=\sum_{\{i, j\}\in [n]^{(2)}}(x_i^2x_j+x_ix_j^2)a_{ij}^2\\
&\phantom{+}+2\sum_{\{i, j, k\}\in[n]^{(3)}} 
	(a_{ij}a_{jk}+a_{jk}a_{ki}+a_{ki}a_{ij})X_{\{i, j, k\}}\\
&=-\Phi+\sum_{\{i, j\}\in [n]^{(2)}}x_ix_ja_{ij}\Big(1-\sum_{k\in [n]-\{i, j\}}x_k\Big)\\
&\phantom{=-\Phi+}+\sum_{\{i, j, k\}\in[n]^{(3)}} (a_{ij}+a_{jk}+a_{ki})X_{\{i, j, k\}}\\
&\phantom{=-\Phi+}-\sum_{T\in [n]^{(3)}}V_TX_T+3\GG(K_3)\\
&=-\Phi-\sum_{T\in [n]^{(3)}}V_TX_T+\gamma+3\GG(K_3)
\end{align*}
we have
\[
\Omega=\lambda\bigl(\gamma+3\GG(K_3)\bigr)-2\gamma\mu
-\lambda\Phi-\lambda\sum_{T\in [n]^{(3)}}V_TX_T\,.
\]

Comparing this with~\eqref{eq:omegalow} we finally arrive at the main estimate 
of this step, namely
\begin{equation}\label{eq:main}
(r-1)\GG(K_r)+(r+1)\GG(K_{r+1})\ls\lambda\bigl(\gamma+3\GG(K_3)\bigr)-2\gamma\mu\,.
\end{equation}

\medskip

{\it Third Step: Introducing and estimating $M$.}

\noindent Let us now define a real number $M$ such that
\begin{equation}\label{eq:mu}
\mu=\frac{(r-2)r}{(s+1)^r}\binom{s+1}{r}(1+\alpha)^{r-1}M\,.
\end{equation}

Since~\eqref{eq:lag-v} yields
\[
r\cdot\GG(K_r)=\sum_{i\in[n]}x_i\GG_i(K_{r-1})
=\sum_{i\in[n]}x_i\bigl(\lambda\GG_i(K_1)-\mu\bigr)=2\gamma\lambda-\mu\,,
\]
we have
\begin{equation}\label{eq:GKR}
\GG(K_r)=\frac 1{(s+1)^r}\binom{s+1}{r}(1+\alpha)^{r-1}\left[1-(r-1)\alpha-(r-2)(M-1)\right]
\end{equation}
owing to~\eqref{eq:lambda} and~\eqref{eq:mu}. 
In view of the presumed smallness of the left-hand side, this gives $M>1$. 
Eventually we shall prove $M\ls 1$ as well, thereby reaching a final contradiction. 
Before doing that, however, we need to provide a much weaker 
upper bound on $M$, so that the results from our fourth section become available. 
This is our next immediate task. To achieve it, we find it convenient to introduce 
the abbreviations
\[
A=\frac r{s(s+1)^{r-1}}\binom{s+1}{r}\,,
\]
\[
B=\frac{(r-2)r}{(s+1)^{r}}\binom{s+1}{r}\,,
\]
and
\[
C=\sqrt[r-2]{\frac{(2\gamma A)^{r-1}}{r\cdot\GG(K_{r})}}\,.
\]
Notice that the last of these stipulations is permissible, as Proposition~\ref{pr31} 
entails ${\GG(K_r)>}0$ in view of $\gamma>\tfrac{s-1}{2s}\gs\tfrac{r-2}{2(r-1)}$. 
Applying the inequality between the arithmetic and geometric mean of $r-1$ positive 
real numbers, one of which is equal to $r\cdot\GG(K_r)$ while each of the remaining 
$r-2$ numbers equals $C(1+\alpha)^{r-1}$, we get
\[
r\cdot\GG(K_r)+(r-2)C(1+\alpha)^{r-1}\gs 2(r-1)A\gamma(1+\alpha)^{r-2}
=2\gamma\lambda=r\cdot\GG(K_r)+\mu
\]
and hence $(r-2)C\gs BM$. Rising both sides to their $(r-2)^{\tn{nd}}$ powers we infer after some easy simplifications
\[
\left(\frac{2\gamma}s\right)^{r-1}\cdot\binom{s+1}{r}\gs (s+1)\GG(K_r)M^{r-2}\,.
\]
Using now Corollary~\ref{cor33}, we obtain
\[
\left(\frac{s-1}s\right)^{r-2}>\frac{s-r+1}{s-1}\cdot M^{r-2}\,.
\]
As this estimate coincides with~\eqref{eq:Msmall}, we now have the results from 
Section~\ref{sec:anna} at our disposal.

\medskip

{\it Fourth Step: Induction on $n$.}

\noindent Due to~\eqref{eq:lag-e} we have for each $i\in[n]$ the inequality
\[
(r-1)\GG_i(K_{r-1})=\sum_{j\in[n]-\{i\}}x_ja_{ij}\GG_{ij}(K_{r-2})\ls 
\sum_{j\in[n]-\{i\}}\lambda x_ja_{ij}=\lambda\GG_i(K_{1})\,.
\]
Hence~\eqref{eq:lag-v} yields
\[
(r-1)\left(\lambda\GG_i(K_{1})-\mu\right)\ls \lambda\GG_i(K_{1})\,,
\]
i.e.,
\[
\GG_i(K_{1})\ls\tfrac{(r-1)\mu}{(r-2)\lambda}=\tfrac s{s+1}(1+\alpha)M\,.
\]
Thus if we define the real number $\eta_i$ to obey 
$\GG_i(K_{1})=\tfrac s{s+1}(1+\alpha)\eta_i$, then $\eta_i\ls M$. 
Similarly but easier we have
\[
0\ls \GG_i(K_{r-1})= \lambda\GG_i(K_{1})-\mu\,,
\]
whence $\GG_i(K_{1})\gs\tfrac{\mu}{\lambda}$, so that altogether we get 
$\eta_i\in \left[\tfrac{r-2}{r-1}\cdot M, M\right]$.

The main objective of this step is to verify that for each $i\in[n]$ one has
\begin{align}\label{eq:window}
\lambda\GG_i(K_2)-&\GG_i(K_r)\ls \frac{(r-2)s(1+\alpha)^r}{(s-1)(s+1)^{r+1}}\binom{s+1}{r}\\
&\times\left(\tfrac 12(r-1)s\vartheta^2-(r-1)s\vartheta M+r(s-1)\eta_i M\right) \notag \,.
\end{align}

Plainly it suffices to show this for $i=n$ and for brevity we are henceforth going to 
write~$\eta$ instead of $\eta_n$. As $\GG_n(K_1)$ is positive, we may construct a 
weighted graph $\GG^{*}$ of order $n-1$ specified by the numbers 
$x_i^*=\tfrac{a_{in}x_i}{\GG_n(K_1)}$ for $i\in[n-1]$ and by the 
restriction of $a$ to~$[n-1]^{(2)}$. Our minimal choices of $r$ and $n$ entail 
$\GG^{*}(K_{r-1})\gs F_{r-1}(\delta)$ and $\GG^{*}(K_{r})\gs F_{r}(\delta)$, 
where $\delta=\GG^{*}(K_2)$. By our construction of $\GG^{*}$ and the case
$i=n$ of~\eqref{eq:lag-v}, we have
\[
\GG^{*}(K_{r-1})=
\frac{\GG_n(K_{r-1})}{\GG_n(K_1)^{r-1}}=\frac{\lambda\GG_n(K_1)-\mu}{\GG_n(K_1)^{r-1}}\,.
\]
Expressing the right-hand side in terms of $r$, $s$, $M$, and $\eta$, we are led 
to ${\GG^{*}(K_{r-1})=H(\eta)}$, where $H$ refers to the function defined just before 
Claim~\ref{Cl41}. Stipulating therefore ${\nu=F_{r-1}^{-1}(H(\eta))}$, as in the hypothesis 
of Claim~\ref{Cl43}, we have $\delta\ls\nu$. Now in view of
\[
\GG^{*}(K_{2})=\frac{\GG_n(K_{2})}{\GG_n(K_1)^{2}} 
\quad \tn{ and } \quad 
\GG^{*}(K_{r})=\frac{\GG_n(K_{r})}{\GG_n(K_1)^{r}}\,,
\]
we get
\[
\lambda\GG_n(K_2)-\GG_n(K_r)
=\frac{s(1+\alpha)^r\eta^2}{(s+1)^r}\times
\left\{(r-1)\binom{s}{r-1}\delta-s^{r-1}\eta^{r-2}\GG^*(K_r)\right\}\,.
\]
By our bound on $\GG^*(K_r)$, the difference in curly braces is at most $Q(\delta)$, 
where~$Q$ signifies the function introduced in Claim~\ref{Cl43},
and by that claim itself this is in turn at most $Q(\nu)$. Estimating now the 
product $\eta^2Q(\nu)$ by means of Claim~\ref{Cl45} we finish proving~\eqref{eq:window}.

\medskip

{\it Fifth Step: Concluding the argument.}

\noindent Notice that~\eqref{eq:deg-sum} and~\eqref{eq:gamma} yield
\[
\frac s{s+1}(1+\alpha)\sum_{i\in[n]}x_i\eta_i
=\sum_{i\in[n]}x_i\GG_{i}(K_1)=2\gamma=\frac s{s+1}(1-\alpha^2)\,,
\]
wherefore
\[
\sum_{i\in[n]}x_i\eta_i=1-\alpha\,.
\]
Thus multiplying~\eqref{eq:window} by $x_i$ and adding up the $n$ resulting inequalities we conclude
\begin{align*}
3\lambda\GG(K_3)&
-(r+1)\GG(K_{r+1})\ls\frac{(r-2)s(1+\alpha)^r}{(s-1)(s+1)^{r+1}}\binom{s+1}{r}\times \\
&\left(\tfrac 12(r-1)s\vartheta^2-(r-1)s\vartheta M+r(s-1)(1-\alpha) M\right)\,.
\end{align*}
Combining this with~\eqref{eq:main} and plugging in the 
formulae~\eqref{eq:gamma},~\eqref{eq:lambda},~\eqref{eq:mu}, and~\eqref{eq:GKR} 
expressing $\gamma$, $\lambda$, $\mu$, and $\GG(K_r)$ in terms of $r$, $s$, $\alpha$, 
and $M$ we get an estimate that on first sight looks rather lengthy. 
After massive cancelations, however, it just reads
\[
(1-\alpha)-2M\ls\frac{(1+\alpha)s^2}{s^2-1}(\vt^2-2M\vt)\,.
\]
Since $\vt\in\left[\frac{s-1}sM, M\right]$ holds by Claim~\ref{Cl42}, we have 
\[
\vt^2-2M\vt=(M-\vt)^2-M^2\ls -\tfrac{s^2-1}{s^2}\cdot M^2\,,
\]
whence
\[
(1-\alpha)-2M\ls-(1+\alpha)M^2\,,
\]
i.e.,
\[
(1-M)((1-\alpha)-(1+\alpha)M)\ls 0\,.
\]
But if $M$ really was greater than $1$, as suggested by our third step, then both factors 
of the left hand side had to be negative. This contradiction finally proves Theorem~\ref{Thm51}.
\end{proof} 

\bigskip

{\bf Acknowledgement.} I would like to thank Roman Glebov, Florian Pfen\-der, 
and Konrad Sperfeld for early discussions about Alexander Razborov's work on
on flag algebras. Further thanks go to L\'{a}szl\'{o}~Lov\'{a}sz and Mikl\'{o}s~Simonovits
for making reference~\cite{LovSim} electronically available, and to all those who have 
volunteered in verifying the proof presented in this paper independently shortly after 
I made it available to them: Konrad Sperfeld and the research group of Tibor Szab\'{o}.
I would like to thank the referees for their
efforts and constructive criticism.
Finally I am indebted to Mathias Schacht for his help with the revision and
to Vojt\v{e}ch R\"{o}dl who translated reference~\cite{Had-Nik} into English.

\end{document}